\documentclass[11pt]{article}

\usepackage[a4paper,margin=1in]{geometry}
\usepackage{graphicx}

\usepackage{amsmath}
\usepackage{amssymb}
\usepackage{amsthm}
\usepackage{mathtools}

\usepackage[T1]{fontenc}
\usepackage{lmodern}
\usepackage{microtype}

\usepackage[hidelinks]{hyperref}

\newtheorem{theorem}{Theorem}
\theoremstyle{remark}
\newtheorem{remark}{Remark}

\title{\bf An Integral Identity Relating Diamond and Square Domains}
\author{Agust\'in Dom\'inguez-Cruz}
\date{January 2026}

\begin{document}
\maketitle

\begin{abstract}
We establish an integral identity for functions on $\mathbb{R}^2$ that are invariant under discrete diagonal translations. The identity shows that integration over the diamond-shaped region $|x|+|y|\le L$ is exactly one half of the integral over the square domain $[-L,L]^2$, allowing diamond-domain integrals to be reduced to easier rectangular integrations.
\end{abstract}

\section*{Main result}

Let
\[
S=[-L,L]^2, \qquad 
D=\{(x,y)\in\mathbb{R}^2:\ |x|+|y|\le L\}, \qquad L>0 .
\]

\begin{theorem}\label{thm:diamond}
Let $f:\mathbb{R}^2\to\mathbb{C}$ satisfy the translation invariances
\[
f(x,y)
=
f(x+L,y+L)
=
f(x+L,y-L)
=
f(x-L,y-L)
=
f(x-L,y+L),
\]
for all $(x,y)\in\mathbb{R}^2$.

Then
\begin{equation}\label{eq:main}
\iint_D f(x,y)\, d x\, d y
=
\frac12
\iint_S f(x,y)\, d x\, d y .
\end{equation}
\end{theorem}

\begin{proof}
Decompose $D$ (up to boundaries of measure zero) into the four congruent right triangles
\[
C_1=\{x\ge0,y\ge0\},\;
C_2=\{x\le0,y\ge0\},\;
C_3=\{x\le0,y\le0\},\;
C_4=\{x\ge0,y\le0\},
\]
intersected with $D$, so that $\int_D f=\sum_{i=1}^4\int_{C_i} f$.

Define the translations
\[
T_1:\ (x,y)\mapsto (x-L,\,y-L),\qquad
T_2:\ (x,y)\mapsto (x+L,\,y-L),
\]
\[
T_3:\ (x,y)\mapsto (x+L,\,y+L),\qquad
T_4:\ (x,y)\mapsto (x-L,\,y+L).
\]
Each $T_i$ has unit Jacobian and preserves $f$ by hypothesis, i.e.\ $f\circ T_i = f$. Moreover, the four images $T_i(C_i)$ tile $S\setminus D$. Therefore
\[
\int_{S\setminus D} f
=\sum_{i=1}^4 \int_{T_i(C_i)} f
=\sum_{i=1}^4 \int_{C_i} f
=\int_D f,
\]
which implies $\int_S f=\int_{S\setminus D} f+\int_D f=2\int_D f$ and proves \eqref{eq:main}.
\end{proof}
For a visual illustration of this proof applied to a specific example, see Fig.~\ref{fig:integrand_domains}.

\remark{An application of this theorem was developed by the author in recent work in mathematical optics \cite{DominguezCruz:2025:BGWigner}.}

\section*{Example}

For any real numbers $A,B,C,D$, consider the integral
\begin{equation}\label{example}
I=\int_{-\pi}^{\pi}\int_{-\pi}^{\pi} 
\exp\!\big[
A\cos(u+v)+B\sin(u+v)+C\cos(u-v)+D\sin(u-v)
\big]\, d u\, d v .
\end{equation}

Write $A\cos t+B\sin t = R_1\cos(t-\delta_1)$ and 
$C\cos t+D\sin t = R_2\cos(t-\delta_2)$, where 
$R_1=\sqrt{A^2+B^2}$, $R_2=\sqrt{C^2+D^2}$, and 
$\delta_1=\arg(A+\mathrm{i}B)$, $\delta_2=\arg(C+\mathrm{i}D)$. Then
\[
I=\int_{-\pi}^{\pi}\int_{-\pi}^{\pi}
e^{R_1\cos(u+v-\delta_1)}\,
e^{R_2\cos(u-v-\delta_2)}\, d u\, d v .
\]

Introduce the change of variables $x=(u+v)/2$, $y=(u-v)/2$, for which
$u=x+y$, $v=x-y$, $d u\, d v = 2\, d x\, d y$, and $[-\pi,\pi]^2$ maps onto
$|x|+|y|\le\pi$. Hence
\begin{equation}\label{eq:I_diamond}
I
=
2\iint_{|x|+|y|\le\pi}
e^{R_1\cos(2x-\delta_1)}\,
e^{R_2\cos(2y-\delta_2)}\, d x\, d y .
\end{equation}

Since the integrand satisfies the translation invariance of
Theorem~\ref{thm:diamond} with $L=\pi$, the integral simplifies as
\[
I=
\int_{-\pi}^{\pi} e^{R_1\cos(2x-\delta_1)}\, d x
\int_{-\pi}^{\pi} e^{R_2\cos(2y-\delta_2)}\, d y.
\]
Because the integrands are $2\pi$--periodic and $\delta_{1,2}$ are defined
modulo $2\pi$, the shifts may be removed by translation. Using
$\int_{-\pi}^{\pi} e^{z\cos\theta}\, d\theta = 2\pi I_0(z)$, we obtain
\begin{equation}\label{eq:IS}
I
=
4\pi^2\, I_0\big(\sqrt{A^2+B^2}\big)\,
I_0\big(\sqrt{C^2+D^2}\big).
\end{equation}

Figure~\ref{fig:integrand_domains} provides a visual illustration of Theorem~\ref{thm:diamond} applied to this example. The three panels show the integrand $g(u,v)$ in \eqref{example} restricted to the domains $S$, $S\setminus D$, and $D$. The dashed lines indicate the decomposition of $D$ into four congruent right triangles, each of which maps onto a corresponding corner of $S\setminus D$ under the translations $T_i$. Intuitively, the triangles in $D$ may be viewed as arrows (with tip at $(0,0)$) indicating the direction in which each $C_i$ is shifted.

\begin{figure}[!t]
    \centering
    \includegraphics[width=0.92\linewidth]{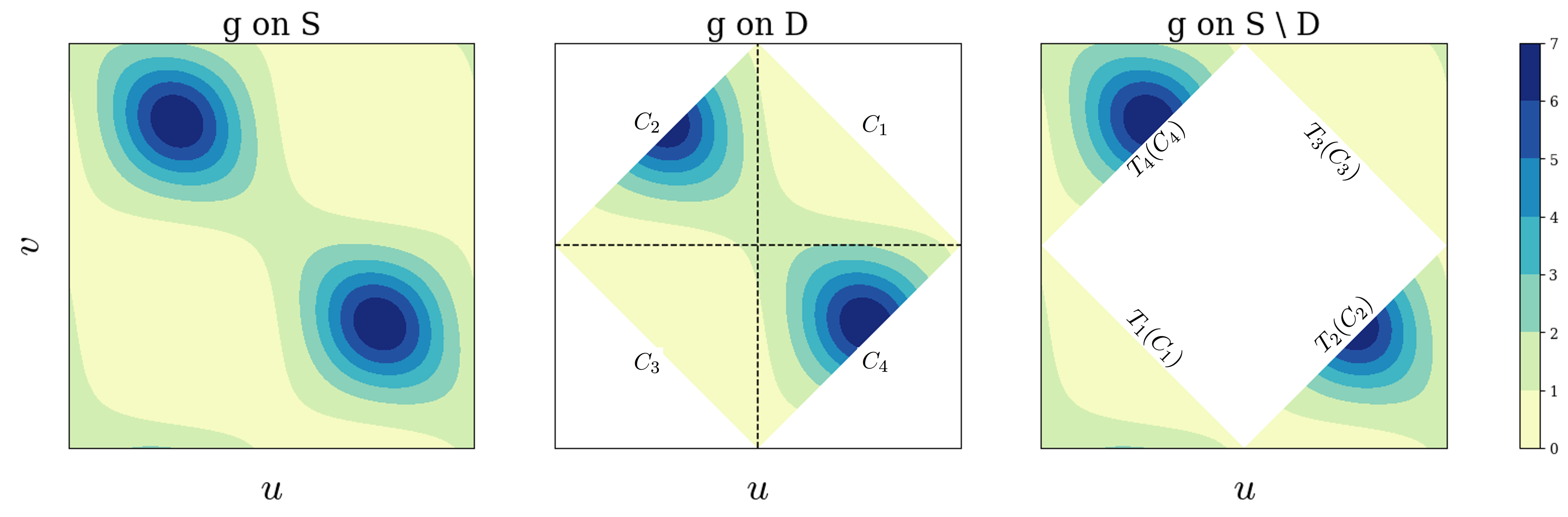}
    \caption{Visualization of the integrand in Eq.~\eqref{example} on the
    domains $S=[-\pi,\pi]^2$, $D=\{(u,v):|u|+|v|\le\pi\}$, and $S\setminus D$,
    for parameters $A=1$, $B=0.5$, $C=-0.8$, and $D=0.2$.}
    \label{fig:integrand_domains}
\end{figure}


\end{document}